\documentclass{amsart}
\usepackage[utf8]{inputenc}
\usepackage{amsmath, amssymb, amsthm, amsfonts, mathtools, array, xfrac, float, multirow, indentfirst, url, enumerate, enumitem, comment, afterpage, pifont, makecell, xcolor, hyperref}
\usepackage[backend=bibtex, sorting=none, bibstyle=numeric, citestyle=numeric, sorting=nyt,maxbibnames=99, giveninits=false, isbn=false, url=false, doi=false, eprint=false]{biblatex}
\usepackage{comment}
\usepackage{thmtools}
\usepackage{cleveref}
\allowdisplaybreaks

\newtheorem{theorem}{Theorem}
\newtheorem{lemma}[theorem]{Lemma}

\newtheorem{corollary}{Corollary}
\newtheorem{remark}{Remark}

\newtheorem{assumption}{Assumption}

\theoremstyle{definition}
\newcommand{\ov}{\overline}
\newcommand{\mc}{\mathcal}
\newcommand{\mr}{\mathrm}
\newcommand{\bb}{\mathbb}
\newcommand{\li}{\mathrm{li}}
\newcommand{\legendre}[2]{\ensuremath{\left( \frac{#1}{#2} \right) }}
\newcommand{\md}{\,\mathrm{d}}
\renewcommand{\Re}{\mathrm{Re}}

\makeatletter
\declaretheoremstyle
    [headformat={\NOTE}, 
    notebraces={}{}, 
    notefont=\bfseries, 
    preheadhook=\def\thmt@space{}, 
    numbered=no
    ]{namedtheorem}
\makeatother

\begin{document}

\author{Tianyu Zhao}

\title{On the mean values of the error terms in Mertens' theorems}

\address{
    Department of Mathematics, The Ohio State University, 231 West 18th
    Ave, Columbus, OH 43210, USA.
}
\email{zhao.3709@buckeyemail.osu.edu}

\subjclass[2020]{11M06, 11M26, 11N05, 11Y35}
\keywords{Mertens' theorems, Riemann hypothesis, Generalized Riemann hypothesis}

\begin{abstract}
    For $i\in \{1,2,3\}$, let $E_i(x)$ denote the error term in each of the three theorems of Mertens on the asymptotic distribution of prime numbers. We show that for $i\in \{1,2\}$ the Riemann hypothesis is equivalent to the condition $\int_2^X E_i(x)\md x>0$ for all $X>2$, and we examine assumptions under which the equivalence also holds for $i=3$. In addition, we extend our results to analogues of Mertens' theorems concerning prime sums twisted by quadratic Dirichlet characters or restricted to arithmetic progressions.
\end{abstract}

\maketitle

\section{Introduction}\label{Introduction}
\subsection{Mertens' theorems}
Set
\[
\mc{E}_1:=-\gamma_E-\sum_p \sum_{n=2}^\infty \frac{\log p}{p^n}=-1.3325\ldots, \hspace{0.5cm} \mc{E}_2:=\gamma_E-\sum_p \sum_{n=2}^\infty \frac{1}{np^n}=0.2614\ldots,
\]
where $\gamma_E=0.5772\ldots$ denotes Euler's constant. Then define
\begin{align*}
    E_1(x):=&\sum_{p\leq x}\frac{\log p}{p}-\log x-\mc{E}_1,\\
    E_2(x):=&\sum_{p\leq x}\frac{1}{p}-\log\log x-\mc{E}_2,\\
    E_3(x):=&\frac{1}{\log x}\prod_{p\leq x}\left(1-\frac{1}{p}\right)^{-1}-e^{\gamma_E}.
\end{align*}
In 1874, Mertens \cite{Mer} showed that $E_1(x)=O(1)$, $E_2(x)=O(1/\log x)$, and $E_3(x)=O(1/\log x)$ as $x\to \infty$. He fell short of proving the prime number theorem, from which one can derive that $E_1(x)=o(1)$ (see \cite[Exercise 8.1.1 \#1]{MV}). Sharper bounds on $E_i(x)$ are available owing to refinements to the prime number theorem. Regarding the oscillations of $E_i(x)$, Robin \cite{Rob} demonstrated that they attain arbitrarily large values in both signs (see also \cite{DP}). However, all values of $E_i(x)$ computed so far are positive; in particular, Rosser and Schoenfeld \cite[Theorems 20, 21, 23]{RS} found that $E_i(x)>0$ for $2\leq x\leq 10^8$. B\"{u}the \cite{JB1} proved that the first sign change of $E_2(x)$ occurs before $x\approx 1.91\cdot 10^{215}$, an analogue of Skewes's number. Such a bias toward positive values has been justified and quantified under certain plausible assumptions by Lamzouri \cite{Lam} based on the seminal work of Rubinstein and Sarnak \cite{RubSar} on the bias exhibited by the error term in the prime number theorem.

One possible way to gain insight into the bias of an oscillating error term is to look at its mean value over finite intervals. Recently Johnston \cite{Joh} established an equivalence between the Riemann hypothesis (RH) and the condition
\[
\int_2^X (\pi(x)-\li(x))\md x<0 \hspace{0.5cm}\text{for all $X>2$}.
\]
See \cite{Suz} and \cite{Was} for related results involving averages of some other error terms that appear in analytic number theory.
In this work we investigate similar behavior exhibited by $E_i(x)$ and discuss what happens in more general contexts. Our first result is as follows.
\begin{theorem}\label{thm E_1 E_2}
For each $i\in \{1,2\}$, the Riemann hypothesis is equivalent to the condition
\[
\int_2^X E_i(x)\md x>0 \hspace{0.5cm}\text{for all $X>2$}.
\]
\end{theorem}

We naturally expect this equivalence to hold for $E_3$ as well. However, complications arise in the hypothetical scenario where $\zeta(s)$ has zeros arbitrarily close to the 1-line, in which case we have to rely on Assumption~\ref{assumption} (see \S\ref{section: proof of thm E_3 Theta=1}). This assumption roughly states that the sharp boundary of the zero-free region can be described by a nice function $\eta$ whose derivative satisfies certain growth conditions. For example, the typical form of zero-free regions established in the literature is given by $\sigma\geq 1-\eta(|t|)$ where $\eta(t)=C(\log t)^{-A}(\log\log t)^{-B}$ (for $t\geq 3$, say) with constants $A,B,C$. Note that $\eta(t)$ satisfies (iii) and (iv) in Assumption~\ref{assumption}, but the problem lies in that we are not certain if the tightest possible bound must also have this form.

We now state the precise result for $E_3(x)$.
\begin{theorem}\label{thm E_3}
Let $\Theta:=\sup\{\Re(\rho):\zeta(\rho)=0\}$.
If $\Theta=1/2$, then $\int_2^X E_3(x)\md x>0$ for all $X>2$. If $1/2<\Theta<1$, or if $\Theta=1$ and Assumption~\ref{assumption} holds, then $\int_2^X E_3(x)\md x$ changes sign infinitely often.
\end{theorem}

The proof of Theorem ~\ref{thm E_1 E_2} relies on the explicit formula and Landau's oscillation theorem. For Theorem ~\ref{thm E_3} we have to carefully examine a race between the first and second moments of $E_2$. In particular, these results show conditionally that on average each $E_i(x)$ is positive on the interval $[2,X]$ for any $X>2$. More precisely, with the definitions
\begin{equation}\label{def f_i}
\begin{split}
    f_1(X):=&\frac{1}{\sqrt{X}}\int_2^X E_1(x)\md x,\\
    f_2(X):=&\frac{\log X}{\sqrt{X}}\int_2^X E_2(x)\md x,\\
    f_3(X):=&\frac{\log X}{e^{\gamma_E} \sqrt{X}}\int_2^X E_3(x)\md x,
\end{split}
\end{equation}
it can be seen from the proofs that
\begin{equation}\label{bounds on f_i}
    2-B_1\leq \liminf_{X\to \infty}f_i(X)< \limsup_{X\to \infty}f_i(X)=2+B_1 
\end{equation}
where $B_1=-2\frac{\xi'}{\xi}(0)=2+\gamma_E-\log 4\pi=0.0461\ldots$. Here $\xi(s)=\frac{s(s-1)}{2}\pi^{-\frac{s}{2}}\Gamma(\frac{s}{2})\zeta(s)$ denotes the completed $\zeta$-function.
\begin{figure}[ht]
    \centering
    \includegraphics[width=9cm]{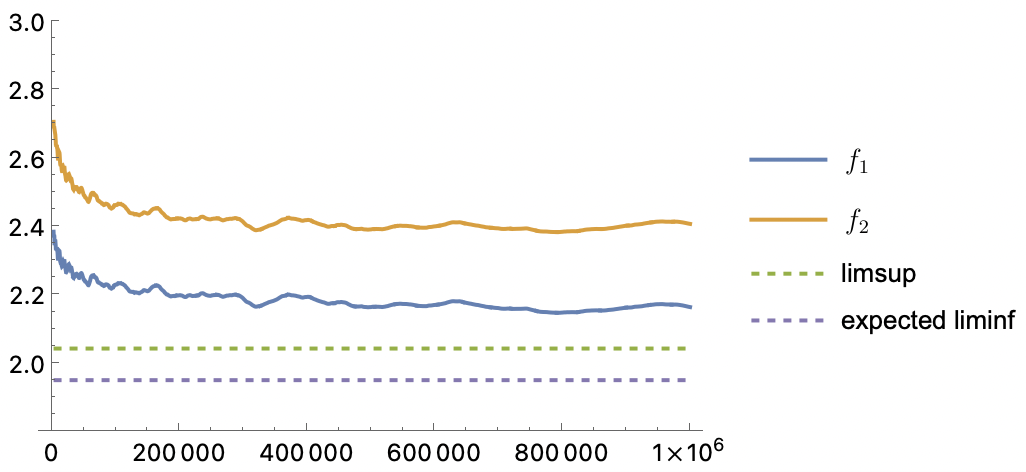}
    \caption{A plot of $f_1(X)$ and $f_2(X)$ for $2000\leq X\leq 10^6$.}
    Note that at $X=10^6$ they are still far above their limsup.
    \label{Figure_1}
\end{figure}
If we additionally assume the Linear Independence conjecture (LI), that is, if the distinct positive ordinates of the non-trivial zeros of $\zeta(s)$ are linearly independent over the rationals (note that non-simple zeros are permitted here), then the first inequality in \eqref{bounds on f_i} can be strengthened to an equality. 

From \eqref{bounds on f_i} it is not hard to derive a localized version of Theorems~\ref{thm E_1 E_2} and \ref{thm E_3}:
\begin{corollary}
    Assuming RH, for each $i\in \{1,2,3\}$ and any positive constant $c<\big(\frac{2-B_1}{2+B_1}\big)^2=0.9548\ldots$,
    \[
    \int_{cX}^X E_i(x)\md x> 0\hspace{0.5cm} \text{for $X\geq X_0(c)$}.
    \]
\end{corollary}

Next we consider the analogue of Mertens' theorems for real primitive Dirichlet characters.

\subsection{Real primitive characters}
For a fundamental discriminant $d$, let $\chi_d:=\legendre{d}{\cdot}$ represent the associated Kronecker symbol. Analogously, one can formulate the Generalized Riemann hypothesis (GRH) and LI for the associated Dirichlet $L$-function $L(s,\chi_d)$. Put
\begin{align*}
    E_1(x;d):=&\sum_{p\leq x}\frac{\chi_d(p)\log p}{p}-\mc{E}_{1}(d),\\
    E_2(x;d):=&\sum_{p\leq x}\frac{\chi_d(p)}{p}-\mc{E}_{2}(d),\\
    E_3(x;d):=&\prod_{p\leq x} \left(1-\frac{\chi_d(p)}{p}\right)^{-1}-L(1,\chi_d),\\
\end{align*}
where
\[
\mc{E}_{1}(d):=-\frac{L'}{L}(1,\chi_d)-\sum_{p}\sum_{n=2}^\infty \frac{\chi_d(p)^n\log p}{p^n},\hspace{0.3cm}\mc{E}_{2}(d):=\log L(1,\chi_d)-\sum_{p}\sum_{n=2}^\infty \frac{\chi_d(p)^n}{np^n}.
\]
One can verify that $E_i(x;d)=o(1)$ as $x\to \infty$.

Let $\mc{D}$ be the union of the two finite sets of fundamental discriminants given in Tables~\ref{table d>0} and ~\ref{table d<0}. We obtain the following analogue of Theorem~\ref{thm E_1 E_2}:
\begin{theorem}\label{thm real prim char}
    For each $i\in \{1,2\}$ and $d\in \mc{D}$, the following are equivalent:
    \begin{enumerate}[label=(\alph*)]
        \item GRH holds for $L(s,\chi_d)$.
        \item $\int_2^X E_i(x;d)\md x>0$ for all sufficiently large $X$.
    \end{enumerate}
    If $d\not \in \mc{D}$ and LI holds for $L(s,\chi_d)$, then $(a)$ implies that $(b)$ is false. In this sense, $\mc{D}$ is the complete set of fundamental discriminants for which the above equivalence holds true. 
\end{theorem}

\begin{remark}\label{remark real prim char}
    It will be noted in the proof of Theorem~\ref{thm real prim char} that the implication $(b)\implies (a)$ is true for every $d$ such that $L(\Theta_{\chi_d},\chi_d)\neq 0$ where $\Theta_{\chi_d}:=\sup\{\Re(\rho): L(\rho,\chi_d)=0\}$. In particular, it is true whenever $L(s,\chi_d)$ has no real non-trivial zero.
\end{remark}

Define $f_i(X;d)$ as in \eqref{def f_i} with the sole change that $e^{\gamma_E}$ is replaced by $L(1,\chi_d)$ in $f_3(X;d)$. We shall see that for each $i\in\{1,2,3\}$ and $d\in \mc{D}$, GRH for $L(s,\chi_d)$ implies that
\begin{equation*}
    2-B_{\chi_d}\leq \liminf_{X\to \infty}f_i(X;d)< \limsup_{X\to \infty}f_i(X;d)=2+B_{\chi_d}, 
\end{equation*}
where $B_{\chi_d}=-2\frac{\xi'}{\xi}(0,\chi_d)$ with $\xi(s,\chi_d)$ denoting the completed $L$-function (see, e.g., \cite[p. 71]{Dav}). Again the first inequality can be replaced by an equality if LI is further assumed.

\begin{remark}
    One may derive an analogue of Theorem~\ref{thm E_3} in the case $\Theta>1/2$, but we shall omit the details.
\end{remark}

Finally, we generalize the preceding discussions to the setting of arithmetic progressions.
\subsection{Arithmetic progressions}
For $q>1$, $(a,q)=1$, let
\begin{align*}
    E_1(x;q,a):=&\sum_{\substack{p\leq x\\p\equiv a\ (\mr{mod}\ q)}}\frac{\log p}{p}-\frac{\log x}{\phi(q)}-\mc{E}_{1}(q,a),\\
    E_2(x;q,a):=&\sum_{\substack{p\leq x\\p\equiv a\ (\mr{mod}\ q)}}\frac{1}{p}-\frac{\log\log x}{\phi(q)}-\mc{E}_{2}(q,a),\\
    E_3(x;q,a):=&\frac{1}{(\log x)^{1/\phi(q)}}\prod_{\substack{p\leq x\\p\equiv a\ (\mr{mod}\ q)}}\left(1-\frac{1}{p}\right)^{-1}-\mc{E}_{3}(q,a),\\
\end{align*}
where $\phi(q)$ denotes Euler's totient function and $\mc{E}_{i}(q,a)$ are appropriate constants that do not affect our discussion \footnote{Languasco and Zaccagnini \cite[\S{6}]{LZ1} gave an elementary expression for $\mc{E}_{3}(q,a)$, which can then be used to express $\mc{E}_{2}(q,a)$ (see also \cite{LZ2}).}. Let
\begin{equation}\label{set Q}
    \mc{Q}=\{2,3,4,5,6,7,8,9,10,12,14,15,16,18,20, 21,24,28,30,36,40,42,48,60\}.
\end{equation}

\begin{theorem}\label{thm arith prog}
    For each $i\in \{1,2\}$ and $q\in \mc{Q}$, the following are equivalent:
    \begin{enumerate}[label=(\alph*)]
        \item GRH holds for $L(s,\chi)$ for each character $\chi\ (\mr{mod}\ q)$;
        \item $\int_2^X E_i(x;q,1)\md x>0$ for all sufficiently large $X$.
    \end{enumerate}
    If $q\not \in \mc{Q}$ and the set 
    \begin{equation}\label{LI for q}
        \bigg\{\gamma:\prod_{\chi\ (\mr{mod}\ q)}L(\beta+i\gamma,\chi)=0, \:0<\beta <1,\:\gamma\geq 0\bigg\}
    \end{equation}
    is linearly independent
    \footnote{Again we only require linear independence among the distinct non-negative ordinates of the non-trivial zeros.} over $\bb{Q}$,
    then $(a)$ implies that $(b)$ is false. In this sense the set $\mc{Q}$ is complete.
\end{theorem}

\begin{remark}
    Similarly as in Remark~\ref{remark real prim char}, the implication $(b)\implies (a)$ is true for every $q$ such that $\prod_{\chi \ (\mr{mod}\ q)}L(\Theta_q,\chi)\neq 0$ where $\Theta_q:=\sup\{\Re(\rho):\prod_{\chi \ (\mr{mod}\ q)}L(\rho,\chi)=0\}$.
\end{remark}

Put $C(q,a):=\#\{b:1\leq b\leq q-1, b^2\equiv a \ (\mr{mod} \ q)\}$. After replacing $e^{\gamma_E}$ by $\mc{E}_{3}(q,a)$ in \eqref{def f_i}, the analogue of \eqref{bounds on f_i} is that for each $i\in\{1,2,3\}$ and $q\in \mc{Q}$, if GRH holds for all $\chi \ (\mr{mod}\ q)$, then
\begin{equation*}
    2C(q,a)-B_q\leq \liminf_{X\to \infty}f_i(X;q,a)<\limsup_{X\to \infty}f_i(X;q,a)\leq 2C(q,a)+B_q, 
\end{equation*}
where 
\[
B_q=-2\frac{\xi'}{\xi}(0)-2 \Re \sum_{\substack{\chi \ (\mr{mod}\ q)\\ \chi\neq \chi_0}}\frac{\xi}{\xi}'(0,\chi^*)
\]
with $\chi^*$ standing for the primitive character inducing $\chi$. The first and the third inequalities both become equalities if the LI condition \eqref{LI for q} holds for $q$.

\subsection{Notations and definitions}
Throughout the text $\rho=\beta+i\gamma$ will denote a non-trivial zero (that is, $0<\beta<1$) of $\zeta(s)$ or $L(s,\chi)$ depending on the context. It is always assumed that $q>1$, $(a,q)=1$ in the context of arithmetic progressions. We list some frequently used counting functions related to primes:
\begin{align*}
    \pi(x)=&\sum_{p\leq x}1, \hspace{0.2cm}\theta(x)=\sum_{p\leq x}\log p, \hspace{0.2cm}\psi(x)=\sum_{n\leq x}\Lambda(n), \hspace{0.2cm}\Pi(x)=\sum_{n\leq x}\frac{\Lambda(n)}{\log n},\\
    \pi(x,\chi)=&\sum_{p\leq x}\chi(p),\hspace{0.2cm} \theta(x,\chi)=\sum_{p\leq x} \chi(p)\log p,\hspace{0.2cm} \psi(x,\chi)=\sum_{n\leq x}\chi(n)\Lambda(n),\\
    \pi(x;q,a)=&\sum_{\substack{p\leq x\\p\equiv a \ (\mr{mod} \ q)}}1,\hspace{0.2cm} \theta(x;q,a)=\sum_{\substack{p\leq x\\p\equiv a \ (\mr{mod} \ q)}}\log p,\hspace{0.2cm} \psi(x;q,a)=\sum_{\substack{n\leq x\\n\equiv a \ (\mr{mod} \ q)}}\Lambda(n),
\end{align*}
where $\Lambda(n)$ is the von Mangoldt function that takes the value $\log p$ when $n$ is a power of the prime $p$ and 0 otherwise. 

To end this section we collect some useful facts that can be found in standard textbooks such as Davenport \cite{Dav} and Montgomery\textendash Vaughan \cite{MV}. 
\begin{lemma}
\begin{align}
    &\psi(x)=x-\sum_{\rho}\frac{x^\rho}{\rho}-\log 2\pi-\frac{1}{2}\log(1-x^{-2}),\hspace{0.3cm}\text{$x>1$ not a prime power};\label{explicit formula}\\
    &\psi(x,\chi)=-\sum_{\rho}\frac{x^{\rho}}{\rho}+O_q(\log x);\label{explicit formula for psi(x,chi)}\\
    &\psi(x)=x+O(\sqrt{x}(\log x)^2),\hspace{0.3cm}\text{assuming RH};\label{error psi(x)-x}\\
    &\psi(x)=\theta(x)+\sqrt{x}+o(\sqrt{x});\label{error psi-theta}\\
    &\psi(x,\chi)=\theta(x,\chi)+\sqrt{x}+o(\sqrt{x});\label{error psi(x,chi)-theta(x,chi)}\\
    &\pi(x)-\li(x)=O(\sqrt{x}\log x)\label{error pi-li RH},\hspace{0.3cm}\text{assuming RH};\\
    &\Pi(x)=\pi(x)+O\left(\frac{\sqrt{x}}{\log x}\right);\label{error Pi-pi}\\
    &\psi(x;q,a)=\frac{1}{\phi(q)}\sum_{\chi \ (\mr{mod}\ q)}\ov{\chi}(a)\psi(x,\chi)\label{orthogonality};\\
    &\psi(x;q,a)=\theta(x;q,a)+C(q,a)\frac{\sqrt{x}}{\phi(q)}+o(\sqrt{x}). \label{error psi-theta arith pro}
\end{align}
\end{lemma}

\section{The mean values of \texorpdfstring{$E_1$}{} and \texorpdfstring{$E_2$}{}: Proof of Theorem~\ref{thm E_1 E_2}, sufficiency}\label{section: proof of thm E_1 E_2 sufficiency}

In this section we prove the sufficiency of RH in Theorem~\ref{thm E_1 E_2}, that is, for each $i\in \{1,2\}$ we show under RH that $\int_2^X E_i(x)\md x>0$ for all $X>2$. Recall that $E_i(x)>0$ for $2\leq x\leq 10^8$, so we  henceforth assume that $X\geq 10^8$. For $i=1$, by partial summation we have
\begin{align*}
    \int_2^X E_1(x)\md x=&\int_2^X \left(\int_{2^-}^x \frac{\mr{d}\theta(y)}{y}\right)\md x-\int_2^X\log x\md x-\mc{E}_1(X-2)\\
    =&\int_{2^-}^X \frac{X-y}{y}\md \theta(y)-(X\log X-X)+(2\log 2-2)-\mc{E}_1(X-2)\\
    =&XE_1(X)-(\theta(X)-X)+2\log 2-2+2\mc{E}_1.
\end{align*}
We can rewrite $E_1(x)$ using (see \cite[(4.21)]{RS})
\begin{equation*}
\sum_{p\leq x}\frac{\log p}{p}=\log x+\mc{E}_1+\frac{\theta(x)-x}{x}-\int_x^\infty \frac{\theta(y)-y}{y^2}\md y,
\end{equation*}
which gives
\begin{align}\label{mean of e_1}
        \int_2^X E_1(x)\md x=-X\int_X^\infty \frac{\theta(x)-x}{x^2}\md x+2\log 2-2+2\mc{E}_1.
\end{align}
The integral above is the tail of the absolutely convergent integral 
\[
\int_2^\infty \frac{\theta(x)-x}{x^2}\md x=\log 2-1+\mc{E}_1=-1.6394\ldots.
\]
From \eqref{explicit formula} and the bound \cite[(3.37)]{RS}
\begin{equation}\label{RS psi-theta}
    \psi(x)-\theta(x)>0.98\sqrt{x}\hspace{0.5cm}\text{for $x\geq 121$},
\end{equation}
we deduce that 
\begin{align*}
    \int_X^\infty \frac{\theta(x)-x}{x^2}\md x<&\int_X^\infty \frac{\psi(x)-x}{x^2}\md x-\frac{1.96}{\sqrt{X}}\\
    <&\sum_{\rho}\frac{1}{\rho(\rho-1)}X^{\rho-1}-\frac{1.96}{\sqrt{X}}\\
    =&-\frac{1}{\sqrt{X}}\sum_{\gamma>0}\frac{2\cos(\gamma \log X)}{1/4+\gamma^2}-\frac{1.96}{\sqrt{X}}\\
    \leq & \frac{B_1-1.96}{\sqrt{X}},
\end{align*}
where
\begin{equation}\label{B_1}
    B_1=\sum_{\gamma }\frac{1}{1/4+\gamma^2}=2\sum_{\rho}\Re\left(\frac{1}{\rho}\right)=0.0461\ldots.
\end{equation}
It follows that $\int_2^X E_1(x)\md x>0$, as desired. We now turn to the case for $E_2$. 

\begin{lemma}\label{lemma for E_2}
Assuming RH,
    \[
    \int_2^x \frac{\theta(u)-u}{u(\log u)^2}\md u<0 \hspace{0.5cm}\text{for all}\hspace{0.2cm}x>2.
    \]
\end{lemma}
\begin{proof}
    Again we may assume that $x\geq 10^8$ since $\theta(x)<x$ for $2\leq x\leq 10^8$ (see \cite[Theorem 18]{RS}) \footnote{The known range of $x$ for which $\theta(x)<x$ has been enlarged to $2\leq x\leq 10^{19}$ in \cite{JB2}.}. We find by applying \eqref{RS psi-theta} that
    \begin{align*}
        \int_2^x \frac{\theta(u)-u}{u(\log u)^2}\md u<&\int_2^x \frac{\psi(u)-u}{u(\log u)^2}\md u-\int_{121}^x \frac{0.98}{\sqrt{u}(\log u)^2}\md u-\int_{2}^{121}\frac{ \psi(u)-\theta(u)}{\sqrt{u}(\log u)^2}\md u\\
        <&\sum_{\rho} \frac{x^\rho}{\rho^2(\log x)^2}-\sum_{\rho}\frac{2^\rho}{\rho^2(\log 2)^2}+2\sum_\rho \int_2^x \frac{u^{\rho-1}}{\rho^2(\log u)^3}\md u\\
        &\hspace{0.5cm}-1.96\frac{\sqrt{x}}{(\log x)^2}+1.96\frac{\sqrt{121}}{(\log 121)^2}-3.92\int_{121}^x \frac{1}{\sqrt{u}(\log u)^3}\md u\\
        \leq& (B_1-1.96)\frac{\sqrt{x}}{(\log x)^2}+B_1\frac{\sqrt{2}}{(\log 2)^2}+2B_1\int_2^{121} \frac{1}{\sqrt{u}(\log u)^3}\md u\\
        &\hspace{0.5cm}+1.96\frac{\sqrt{121}}{(\log 121)^2}+(2B_1-3.92)\int_{121}^x \frac{1}{\sqrt{u}(\log u)^3}\md u\\
        <&(B_1-1.96)\frac{\sqrt{x}}{(\log x)^2}.
\end{align*}
\end{proof}

In the same vein as the earlier argument for $E_1$, the formula
\begin{equation}\label{sum of 1/p}
    \sum_{p\leq x}\frac{1}{p}=\log\log x+\mc{E}_2+\frac{\pi(x)-\li(x)}{x}-\int_x^\infty \frac{\pi(y)-\li(y)}{y^2}\md y
\end{equation}
(which can be seen, e.g., by applying \cite[(2.27)]{RS}) yields
\begin{align}\label{mean of e_2}
    \int_2^X E_2(x)\md x=&XE_2(X)-(\pi(X)-\li(X))+(2\mc{E}_2+2\log\log 2-\li(2))\notag\\
    =&-X\int_X^\infty \frac{\pi(x)-\li(x)}{x^2}\md x+(2\mc{E}_2+2\log\log 2-\li(2)).
\end{align}
The integral in question is the tail of
\[
\int_2^\infty \frac{\pi(x)-\li(x)}{x^2}\md x=\mc{E}_2+\log\log 2-\frac{\li(2)}{2}=-0.6275\ldots.
\]
Partial summation gives (see, e.g., \cite[(13.2)]{MV})
\begin{align}\label{pi-li interms of theta-x}
    \pi(x)-\li(x)=&\frac{\theta(x)-x}{\log x}+\int_2^x \frac{\theta(u)-u}{u(\log u)^2}\md u+\frac{2}{\log 2}-\li(2)\\
    <&\frac{\psi(x)-x}{\log x}-\frac{0.98\sqrt{x}}{\log x}\notag
\end{align}
via \eqref{RS psi-theta} and Lemma~\ref{lemma for E_2}. Consequently, 
\begin{align*}
    \int_X^\infty \frac{\pi(x)-\li(x)}{x^2}\md x<&\int_X^\infty \frac{\psi(x)-x}{x^2\log x}\md x-\frac{1.96}{\sqrt{X}\log X}+\int_X^\infty \frac{1.96}{x^{3/2}(\log x)^2}\md x\\
    <& \sum_\rho \frac{1}{\rho(\rho-1)}\frac{X^{\rho-1}}{\log X}-\sum_\rho \frac{1}{\rho(\rho-1)}\int_X^\infty\frac{x^{\rho-2}}{(\log x)^2}\md x\\
    &\hspace{1cm}-\frac{1.96}{\sqrt{X}\log X}+\int_X^\infty \frac{1.96}{x^{3/2}(\log x)^2}\md x\\
    \leq & \frac{B_1-1.96}{\sqrt{X}\log X}+\int_X^\infty\frac{B_1+1.96}{x^{3/2}(\log x)^2}\md x\\
    <& -\frac{1.73}{\sqrt{X}\log X},
\end{align*}
and so $\int_2^X E_2(x)\md x>0$. This concludes the proof. 

Restating these arguments asymptotically with \eqref{error psi-theta} in mind gives \eqref{bounds on f_i} for $i\in \{1,2\}$. The equality in \eqref{bounds on f_i} follows from Dirichlet's theorem; assuming LI, Kronecker's theorem implies that the first inequality can be replaced by an equality as well.

\section{The mean values of \texorpdfstring{$E_1$}{} and \texorpdfstring{$E_2$}{}: Proof of Theorem~\ref{thm E_1 E_2}, necessity}\label{section: proof of thm E_1 E_2 necessity}

Next we prove the necessity of RH in Theorem~\ref{thm E_1 E_2}. More precisely, if RH is false, we shall see that $\int_2^X E_i(x)\md x$ attains arbitrarily large positive and negative values for $i=1,2$. We start with a standard lemma on the values of two special integrals.

\begin{lemma}\label{lemma necessity for thm1}
\begin{enumerate}[label=(\alph*)]
    \item\label{a} \[
    \int_1^\infty \frac{\psi(x)-x}{x^2}\md x=-1-\gamma_E.
    \]
    \item\label{b}
    \[
    \int_2^\infty \frac{\Pi(x)-\li(x)}{x^2}\md x=\gamma_E+\log\log 2.
    \]
\end{enumerate}
\end{lemma}

\begin{proof}
    \begin{enumerate}[label=(\alph*)]
        \item It is enough to recall that
        \[
        \int_1^\infty \frac{\psi(x)-x}{x^{s+1}}\md x=-\frac{\zeta'(s)}{s\zeta(s)}-\frac{1}{s-1},\hspace{0.5cm}{\sigma>1}
        \]
        where
        \[\frac{\zeta'(s)}{\zeta(s)}=-\frac{1}{s-1}+\gamma_E+O(|s-1|).
        \]

        \item 
        For $\sigma>1$ we have
        \[
        \int_2^\infty \frac{\Pi(x)}{x^{s+1}}\md x=\frac{\log\zeta(s)}{s}\hspace{0.5cm}\text{and}\hspace{0.5cm} \int_2^\infty \frac{\li(x)}{x^{s+1}}\md x=\frac{-\log(s-1)+r(s)}{s}
        \]
        where 
        \[
        r(s):=-\int_0^{(s-1)\log 2}\frac{e^{-u}-1}{u}\md u-\gamma_E-\log\log 2
        \]
        is entire (see \cite[Theorem 15.2]{MV}). Thus,
        \[
        \int_2^\infty \frac{\Pi(x)-\li(x)}{x^2}\md x=\lim_{s\to 1^+}\left(\frac{\log(\zeta(s)(s-1))}{s}-\frac{r(s)}{s}\right)=\gamma_E+\log\log 2.
        \]
    \end{enumerate}
\end{proof}

Supposing that $\Theta>1/2$, we show for $i\in\{1,2\}$ that
\begin{equation}\label{Omega bound for E_1 E_2}
    \int_2^X E_i(x)\md x=\Omega_{\pm}(X^{\Theta-\epsilon})
\end{equation}
where the notation $f=\Omega_{\pm}(g)$ means $\limsup_{x\to \infty}\frac{f(x)}{|g(x)|}>0$ and $\liminf_{x\to \infty}\frac{f(x)}{|g(x)|}<0$, as usual. In view of \eqref{error psi-theta}, \eqref{error Pi-pi}, \eqref{mean of e_1} and \eqref{mean of e_2}, it suffices to establish
\begin{equation}\label{suff condition on psi(x)-x}
    \int_X^\infty \frac{\psi(x)-x}{x^2}\md x=\Omega_{\pm}(X^{\Theta-1-\epsilon})
\end{equation}
and
\begin{equation}\label{suff condition on Pi(x)-x}
    \int_X^\infty \frac{\Pi(x)-\li(x)}{x^2}\md x=\Omega_{\pm}(X^{\Theta-1-\epsilon}).
\end{equation}
First let
\begin{align*}
    A_1(X):=X^{\Theta-1-\epsilon}+\int_X^\infty \frac{\psi(x)-x}{x^2}\md x.
\end{align*}
For $\sigma>1$, using integration by parts and Lemma~\ref{lemma necessity for thm1} we obtain
\begin{align*}
    F(s):=&\int_1^\infty A_1(x) x^{-s}\md x\\
    =&\frac{1}{s-\Theta+\epsilon}+\frac{1}{s-1}\int_1^\infty \frac{\psi(x)-x}{x^2}\md x+\frac{1}{s-1}\int_1^\infty  \frac{-\psi(x)+x}{x^{s+1}}\md x\\
    =&\frac{1}{s-\Theta+\epsilon}+\frac{1}{s-1}\left(-1-\gamma_E+\frac{\zeta'(s)}{s\zeta(s)}+\frac{1}{s-1}\right).
\end{align*}
Observe that this function is analytic at all real values $s>\Theta-\epsilon$. In particular, it is analytic at $s=1$ as
\[
F(s)=\frac{1}{s-\Theta+\epsilon}+\frac{1}{s-1}\left(-\frac{(s-1)(1+\gamma_E)}{s}+O\left(\bigg|\frac{s-1}{s}\bigg|\right)\right)
\]
as $s\to 1$. If $A_1(X)>0$ for all sufficiently large $X$, then a standard application of Landau's oscillation theorem (again see the proof of \cite[Theorem 15.2]{MV}) leads to a contradiction. This proves the $\Omega_{-}$ part of \eqref{suff condition on psi(x)-x}, and the $\Omega_{+}$ part can be deduced in a similar way.

To prove \eqref{suff condition on Pi(x)-x}, consider $\int_2^\infty A_2(x) x^{-s}\md x$ where
\begin{align*}
    A_2(X):=X^{\Theta-1-\epsilon}+\int_X^\infty \frac{\Pi(x)-\li(x)}{x^2}\md x
\end{align*}
and then proceed as before. The proof of Theorem~\ref{thm E_1 E_2} is now complete.

\section{The mean value of \texorpdfstring{$E_3$}{}: Proof of Theorem~\ref{thm E_3}, \texorpdfstring{$1/2\leq \Theta<1$}{}} \label{section: proof of thm E_3 Theta<1}

The key to studying the mean of $E_3$ turns out to be the mean square of $E_2$, which will be addressed in Lemma~\ref{lemma mean square E_2} below. We remark that the full strength of this result will not necessarily be needed; in fact, for our purposes it suffices to trivially bound $\int_2^X (E_2(x))^2\md x$ by inserting classical bounds such as
\begin{equation}\label{error PNT Theta}
    \psi(x)-x=O(x^\Theta(\log x)^2)\hspace{0.3cm}\text{and} \hspace{0.3cm} \pi(x)-\li(x)=O(x^\Theta\log x)
\end{equation}
into the integral. However, we choose to include the lemma for completeness. Regarding the estimates \eqref{error PNT Theta}, in fact Grosswald \cite{Gro} showed that a factor of $(\log x)^2$ can be saved if $\Theta>1/2$. That is,
\begin{equation}\label{Grosswald}
    \psi(x)-x=O_\Theta(x^\Theta)\hspace{0.3cm}\text{and} \hspace{0.3cm} \pi(x)-\li(x)=O_\Theta(x^\Theta/\log x),
\end{equation}
with the implied constants now depending on $\Theta$. The proof exploits Ingham's zero-density estimate \cite{Ing}, which shows that the contribution to the explicit formula from zeros with real part close to $\Theta$ is $O(1)$.

Before proving Lemma~\ref{lemma mean square E_2}, we first recall a mean square estimate of $\psi(x)-x$. 

\begin{lemma}\label{lemma mean square psi(x)-x}
If $\Theta=1/2$, then 
\[
\int_X^{2X} (\psi(x)-x)^2 \md x\asymp X^2.
\]

If $\Theta>1/2$, then
\[
X^{2\Theta+1-\epsilon}\ll \int_X^{2X} (\psi(x)-x)^2 \md x\ll X^{2\Theta+1},
\]
where the upper bound is effective and uniform in $\Theta$, while the lower bound is ineffective.
\end{lemma}

\begin{proof}
    \textit{Case 1}: $\Theta=1/2$. For the upper bound see the proof of \cite[Theorem 13.5]{MV}. The idea is to expand the square of the explicit formula and then show that
    \begin{equation}\label{double sum over zeros}
        \sum_{\gamma_i,\gamma_j}\frac{1}{|\gamma_i\gamma_j|(1+|\gamma_i-\gamma_j|)}<\infty
    \end{equation}
    by appropriately splitting the sum over $\gamma_j$ for each fixed $\gamma_i$. In the other direction, Stechkin and Popov \cite[Theorem 6]{PS} proved that
    \[
    \int_X^{2X}|\psi(x)-x|\md x\gg X^{3/2},
    \]
    which gives the desired lower bound via the Cauchy\textendash Schwarz inequality. \footnote {We refer the reader to \cite{BPT} for sharpened estimates.}

    \textit{Case 2}: $\Theta>1/2$. The upper bound follows from a trivial modification of the proof of \cite[Theorem 13.5]{MV}. (Note that if we integrate using the pointwise bound \eqref{Grosswald}, the implied constant would depend on $\Theta$.) The lower bound is a consequence of \cite[Theorem 7]{PS}:
    \[
    \int_X^{2X}|\psi(x)-x|\md x\gg X^{1+\Theta-\epsilon}
    \]
    where the implied constant is ineffective in that it depends on the hypothetical zero distribution of $\zeta(s)$ near $\sigma=\Theta$.
\end{proof}

\begin{lemma}\label{lemma mean square E_2}
\[
\int_{2}^{X} (E_2(x))^2\md x=
\begin{cases}
    O(1) &\text{if $\Theta=1/2$},\\
    O_\Theta\left(\frac{X^{2\Theta-1}}{(\log X)^2}\right) &\text{if $\Theta>1/2$}.
\end{cases}
\]
\end{lemma}

\begin{proof}
We see from \eqref{sum of 1/p} that
\[
\int_2^X (E_2(x))^2\md x
\ll \int_2^X \frac{(\pi(x)-\li(x))^2}{x^2}\md x+\int_2^X \left(\int_x^\infty \frac{\pi(t)-\li(t)}{t^2}\md t\right)^2\md x.
\]
Denote these two integrals by $I_1$ and $I_2$, respectively.

\textit{Case 1}: $\Theta=1/2$. On combining \eqref{error psi(x)-x}, \eqref{RS psi-theta} and \eqref{pi-li interms of theta-x}, we find that in this case
\[
    \pi(x)-\li(x)=\frac{\psi(x)-x}{\log x}-\frac{\sqrt{x}}{\log x}+O\left(\frac{\sqrt{x}}{(\log x)^2}\right),
\]
so that
\begin{align*}
    I_1=\int_2^X \frac{(\psi(x)-x)^2}{x^2(\log x)^2}\md x+O\left(\int_2^X \frac{(\psi(x)-x)\sqrt{x}}{x^2(\log x)^2}\md x\right)+O\left(\int_2^X \frac{(\sqrt{x})^2}{x^2(\log x)^2}\md x\right).
\end{align*}
The third integral is plainly $O(1)$, and so is the second integral as shown by termwise integration of the explicit formula. For the first integral, we apply Lemma~\ref{lemma mean square E_2} and sum over dyadic intervals. Moreover, $I_2=O(1)$ as a consequence of the estimate 
\begin{equation}
    \int_x^\infty \frac{\pi(t)-\li(t)}{t^2}\md x\ll \frac{1}{\sqrt{x}\log x},
\end{equation}
which can be proved similarly to the argument at the end of $\S \ref{section: proof of thm E_1 E_2 sufficiency}$.

\bigskip
\textit{Case 2: $\Theta>1/2$}. 
Bounding $I_1$ and $I_2$ with \eqref{Grosswald} suffices here. A more elementary way (free of using any zero-density estimate) is to rewrite both integrals using \eqref{pi-li interms of theta-x} and work with the explicit formula again. However, either way the implicit constant depends on $\Theta$ and becomes worse when $\Theta$ approaches $1/2$. For example, for the first integral appearing in the expansion of $I_1$, Lemma~\ref{lemma mean square psi(x)-x} shows that 
\[
\int_X^{2X} \frac{(\psi(x)-x)^2}{x^2(\log x)^2}\md x\ll \frac{X^{2\Theta-1}}{(\log X)^2},
\]
and then summing over dyadic intervals gives a constant of size approximately $1/(2\Theta-1)$. We now present an alternative argument of similar nature, which will also be useful for Lemma~\ref{lemma int E_2^2 by zero-free region}. On recalling the truncated version of the explicit formula (see, e.g., \cite[p. 109]{Dav}), we observe that for any $0<\delta<\Theta-1/2$ this integral is bounded by 
\begin{align*}
    \sum_{|\gamma_i|,|\gamma_j|\leq X^2}&\frac{1}{\rho_i\ov{\rho_j}}\int_2^X \frac{x^{\beta_i+\beta_j-2+i(\gamma_i-\gamma_j)}}{(\log x)^2}\md x+O(1)\\
    =& \sum_{\substack{\beta_j\geq 1/2+\delta\\ |\gamma_i|,|\gamma_j|\leq X^2}}\frac{1}{\rho_i\ov{\rho_j}}\int_2^X \frac{x^{\beta_i+\beta_j-2+i(\gamma_i-\gamma_j)}}{(\log x)^2}\md x\\
    &\hspace{1cm}+O\left(\int_2^X \sum_{|\gamma_i|\leq X^2}\frac{x^{\beta_i}}{|\rho_i|}\sum_{\substack{\beta_j<1/2+\delta\\ |\gamma_j|\leq X^2}} \frac{x^{\beta_j}}{|\rho_j|}\frac{1}{x^2(\log x)^2}\md x\right)\\
    \ll & \frac{X^{2\Theta-1}}{(\log X)^2}\sum_{\gamma_i,\gamma_j}\frac{1}{|\gamma_i\gamma_j(\delta+|\gamma_i-\gamma_j|)|}+O\left(X^{\Theta+\delta-1/2}(\log X)^2\right)\\
    \ll&_\Theta \frac{X^{2\Theta-1}}{(\log X)^2},
\end{align*}
where the double sum in the penultimate line converges in comparison with \eqref{double sum over zeros}. The other integrals can be treated similarly. 
\end{proof}

\begin{remark}
    The proof of \cite[Theorems 7 and 12]{PS} can most likely be modified to show that
    \[
    \int_2^X |E_2(x)|\md x\gg X^{\Theta-\epsilon},
    \]
    which in turn implies that
    \[
    \int_2^X (E_2(x))^2\md x \gg X^{2\Theta-1-\epsilon}.
    \]
    We do not pursue this here. Note that from \eqref{Omega bound for E_1 E_2} one can only assert this for a sequence of $X\to \infty$.   
\end{remark}

We are now ready to prove Theorem~\ref{thm E_3} in the case where $1/2\leq \Theta<1$. Note that
\begin{align*}
    \int_2^X E_3(x)\md x=&\int_2^X \left(\exp\left(-\log\log x-\sum_{p\leq x}\log (1-p^{-1})\right)-e^{\gamma_E}\right)\md x\\
    =&\int_2^X \left(\exp\left(-\log\log x+\sum_{p\leq x}p^{-1}-\mc{E}_2+\gamma_E\right)-e^{\gamma_E}\right)\md x\\
    =&e^{\gamma_E} \int_2^X \left(e^{E_2(x)}-1\right)\md x\\
    =&e^{\gamma_E} \int_2^X E_2(x)\md x+e^{\gamma_E}\sum_{n=2}^\infty \frac{1}{n!}\int_2^X (E_2(x))^n\md x.
\end{align*}
Under RH, the first integral in the last line is positive by Theorem~\ref{thm E_1 E_2}. Moreover,
\begin{equation}\label{bound higher moments of E_2}
    \frac{1}{2}\int_2^X (E_2(x))^{2}\md x> \sum_{n=3}^\infty \frac{1}{n!}\int_2^X |E_2(x)|^n\md x
\end{equation}
because $|E_2(x)|<1$ for all $x>2$ (see \cite[Theorem 5]{RS}), and consequently the sum of remaining integrals is also positive. On the other hand, if RH is false with $1/2<\Theta<1$, then Lemma~\ref{lemma mean square E_2} together with \eqref{bound higher moments of E_2} implies that
\[
\sum_{n=2}^\infty \frac{1}{n!}\int_2^X (E_2(x))^n\md x\ll X^{2\Theta-1}.
\]
On invoking \eqref{Omega bound for E_1 E_2} and choosing $\epsilon<1-\Theta$, we find that 
\[
\int_2^X E_3(x)\md x=\Omega_{\pm} (X^{\Theta-\epsilon}).
\]

However, the previous line of reasoning breaks down if $\Theta=1$ since $\int_2^X (E_2(x))^2\md x$ and the large values of $\int_2^X E_2(x)\md x$ are of comparable magnitude, making it unclear whether there still exists arbitrarily large $X$ where the first moment dominates the second moment. Our next step is to address this issue using tools from the works \cite{Pin1, Pin2} of Pintz.

\section{The mean value of \texorpdfstring{$E_3$}{}: Proof of Theorem~\ref{thm E_3}, \texorpdfstring{$\Theta=1$}{}} \label{section: proof of thm E_3 Theta=1}

For simplicity of notation, put $\Delta_1(x):=\int_2^x E_2(t)\md t$ and $\Delta_2(x):=\int_2^x (E_2(t))^2\md t$, so that 
\[
\Delta_1(x)< e^{-\gamma_E}\int_2^x E_3(t)\md t< \Delta_1(x)+\Delta_2(x)
\]
in view of \eqref{bound higher moments of E_2}. We already know from \eqref{Omega bound for E_1 E_2} that there exists arbitrarily large $x$ with $\Delta_1(x)>0$, so our goal is to show that $\Delta_1(x)<-\Delta_2(x)$ infinitely often. We first need a more precise result on the oscillations of $\Delta_1(x)$.
\begin{lemma}\label{lemma oscill caused by single zero}
    Let $\epsilon>0$. If $\zeta(s)$ has a zero $\rho_0=\beta_0+i\gamma_0$ where $\beta_0>1/2+\epsilon$ and $\gamma_0>C_1(\epsilon)$, then for all $X>\gamma_0^{C_2(\epsilon)}$ there exist $x_1,x_2\in [X,X^{1+\epsilon}]$ such that
    \begin{equation}\label{eqn oscill caused by single zero}
        \Delta_1(x_1)>\frac{x_1^{\beta_0}}{\gamma_0^{2+\epsilon}\log x_1} \hspace{0.3cm}\text{and} \hspace{0.3cm}
        \Delta_1(x_2)<-\frac{x_2^{\beta_0}}{\gamma_0^{2+\epsilon}\log x_2}.
    \end{equation}
    Both $C_1(\epsilon)$ and $C_2(\epsilon)$ are effectively computable.
\end{lemma}

We only provide a sketch of proof as it follows the same lines as that of \cite[Theorem 2]{Pin1} where Pintz showed (under the same conditions) that
\[
    \pi(x_1)-\li(x_1)>\frac{x_1^{\beta_0}}{\gamma_0^{1+\epsilon}\log x_1} \hspace{0.3cm}\text{and} \hspace{0.3cm}
    \pi(x_2)-\li(x_2)<-\frac{x_2^{\beta_0}}{\gamma_0^{1+\epsilon}\log x_2}.
\]
Heuristically speaking, in our setting the exponent 2 in the denominators comes from the fact that $\int_x^\infty  (\psi(t)-t)/t^2 \md t\approx \sum_\rho \frac{x^{\rho-1}}{\rho(\rho-1)}\approx \sum_\rho \frac{x^{\rho-1}}{\gamma^2}$, which should oscillate up to the same magnitude (modulo a factor of $\log x$) as $f(x):=\int_x^\infty (\pi(t)-\li(t))/t^2\md t$, which is connected to $\Delta_1(x)$ via \eqref{mean of e_2}. 
    
\begin{proof}
    We work with the function
    \[
    g(x):=\int_x^\infty \frac{\Pi(t)-\mr{lg}(t)}{t^2}\md t
    \]
    where 
    \[
    \mr{\lg}(t):=
    \begin{cases}
        0, & 0<x<2,\\
        \displaystyle \sum_{2\leq n\leq x}\frac{1}{\log n}, & x\geq 2.
    \end{cases}
    \]
    It is clear that $\mr{\lg}(t)=\li(t)+O(1)$. In conjunction with \eqref{error Pi-pi}, this gives $f(x)=g(x)+O(x^{-1/2}(\log x)^{-1})$. Further define $H(s):=\zeta'(s)/\zeta(s)+\zeta(s)-1$. It is not hard to show that
    \begin{equation}\label{H(s)}
        \int_1^\infty g(x)\frac{\mr{d}}{\mr{d}x}\left(sx^{-s+1}\log x-x^{-s+1}\right)\md x=H(s),\hspace{0.5cm}\sigma>1.
    \end{equation}
    After suitably choosing another zero $\gamma_1$ close to $\gamma_0$ and parameters $k$, $\mu$, we compute the integral
    \[
    U=\frac{1}{2\pi i}\int_{2-i\infty}^{2+i\infty}H(s+i\gamma_1)e^{ks^2+\mu s}\md x
    \]
    in two ways, one by invoking \eqref{H(s)} and switching the order of integration, and the other by moving the line of integration and applying the residue theorem. The power sum method is then used to yield a contradiction if $g(x)$ fails to get sufficiently large in both signs in the interval. We briefly mention the changes to the first way. In comparison with \cite[(10.3)]{Pin1}, here we have
    \begin{align*}
        U&=\frac{1}{2\pi i}\int_{2-i\infty}^{2+i\infty}\int_1^\infty g(x)\frac{\mr{d}}{\mr{d}x}\left(\left((s+i\gamma_1)x^{-s-i\gamma_1+1}\log x-x^{-s-i\gamma_1+1}\right)e^{ks^2+\mu s}\right)\md x\md s\\
        &=\int_1^\infty g(x)\frac{\mr{d}}{\mr{d}x}\left(x^{-i\gamma_1}x\log x\cdot \frac{1}{2\pi i}\int_{2-i\infty}^{2+i\infty}se^{ks^2+(\mu-\log x)s}\md s\right)\md x\\
        &\hspace{1cm}+i\int_1^\infty g(x)\frac{\mr{d}}{\mr{d}x}\left(\gamma_1 x^{-i\gamma_1}x\log x\cdot \frac{1}{2\pi i}\int_{2-i\infty}^{2+i\infty}e^{ks^2+(\mu-\log x)s}\md s\right)\md x\\
        &\hspace{2cm}-\int_1^\infty g(x)\frac{\mr{d}}{\mr{d}x}\left(x^{-i\gamma_1}x\cdot \frac{1}{2\pi i}\int_{2-i\infty}^{2+i\infty}e^{ks^2+(\mu-\log x)s}\md s\right)\md x\\
        &=: I_1+I_2+I_3.
    \end{align*}
    The following identity related to \cite[(10.2)]{Pin1} is useful in evaluating $I_1$:
    \begin{align*}
        &\frac{1}{2\pi i}\int_{2-i\infty}^{2+i\infty}se^{As^2+Bs}\md s=-\frac{1}{2\pi i}\frac{B}{2A}\int_{2-i\infty}^{2+i\infty}e^{As^2+Bs}\md s=-\exp\left(-\frac{B^2}{4A}\right)\frac{B}{4\sqrt{\pi}A^{3/2}}
    \end{align*}
    for any $A>0$, $B\in \bb{C}$. The term $\gamma_1^2$ appears when one performs differentiation inside $I_2$, and is ultimately responsible for the exponent $2$ in \eqref{eqn oscill caused by single zero}. 
\end{proof}

On the other hand, $\Delta_2(x)$ is controlled by the zero-free region of $\zeta(s)$ in a similar manner to the error term in the prime number theorem (cf. \cite{Pin2}).
\begin{lemma}\label{lemma int E_2^2 by zero-free region}
    Suppose that $\zeta(s)$ has no zero (or finitely many zeros if $\Theta=1$) in the region $\sigma>1-\eta(|t|)$ where $\eta(t):[0,\infty)\to (0,1/2]$ is a continuous decreasing function. Fix $0<\epsilon<1$ and define $\omega(x):=\min_{t\geq 1}(\eta(t)\log x+\log t)$. Then 
    \[
    \Delta_2(x)=O_\epsilon\left(\frac{x}{e^{2(1-\epsilon)\omega(x)}(\log x)^2}\right).
    \]
\end{lemma}
\begin{proof}
    We may assume that $\Theta=1$ as the result follows readily from Lemma~\ref{lemma mean square E_2} if $\Theta<1$. From the last argument in the proof of Lemma~\ref{lemma mean square E_2}, we see that for any fixed $0<\delta<1/2$,
    \begin{align*}
        \Delta_2(x)\ll \sum_{|\gamma_i|,|\gamma_j|\leq x^2}\frac{1}{|\gamma_i\gamma_j(\delta+|\gamma_i-\gamma_j|)|} \frac{x^{\beta_i+\beta_j-1}}{(\log x)^2}
        \ll \frac{1}{x(\log x)^2}\left(\sum_{|\gamma|\leq x^2}\frac{x^{\beta}}{|\gamma|}\right)^2.
    \end{align*}
    The sum in the parenthesis is at most $O_\epsilon(x/e^{(1-\epsilon)\omega(x)})$ according to the proof of \cite[Theorem 1]{Pin2}. This proves the lemma.
\end{proof}

We now prove Theorem~\ref{thm E_3} for the case $\Theta=1$ under the following assumption:

\begin{assumption}\label{assumption}
    There exists a differentiable function $\eta(t):[0,\infty)\to (0,1/2]$ with the following properties:
    \begin{enumerate}[label=(\roman*)]
        \item $\zeta(s)$ has finitely many zeros in the region $\sigma>1-\eta(|t|)$.

        \item $\zeta(s)$ has infinitely many zeros in the region $\sigma>1-K \eta(|t|)$ for some constant $K>1$.

        \item $\eta'(t)t$ strictly increases to 0 as $t\to \infty$.

        \item $\eta(t)/|\eta'(t)|\gg t\log t$.
    \end{enumerate}
\end{assumption}
    
Without loss of generality we may suppose that $1<K<2$, since otherwise we can simply multiply $\eta(t)$ by a suitable constant. Let $\epsilon=\epsilon(K)$ be a small positive constant to be determined later. Take a zero $\rho_0=\beta_0+i\gamma_0$ with $\gamma_0>C_1(\epsilon)$ (as defined in Lemma~\ref{lemma oscill caused by single zero}) sufficiently large and $\beta_0>1-K\eta(\gamma_0)>1/2+\epsilon$. Such $\rho_0$ exists and there are infinitely many of them due to assumption (ii). 

For any fixed $x>1$ the minimum of $\eta(t)\log x+\log t$ occurs at some $t=t_0$ so that $\log x=-1/(\eta'(t_0)t_0)$. It follows from assumption (iii) that such $t_0$ is unique if exists. Thus if we choose $X_0$ such that $\log X_0=1/(|\eta'(\gamma_0)|\gamma_0)$, then 
\[
\omega(X_0)=\frac{\eta(\gamma_0)}{|\eta'(\gamma_0)| \gamma_0}+\log \gamma_0.
\]
By assumption (iv) $X_0\geq  \gamma_0^{c/\eta(\gamma_0)}>\gamma_0^{C_2(\epsilon)}$, so Lemma~\ref{lemma oscill caused by single zero} implies that there exists $x_2\in [X_0,X_0^{1+\epsilon}]$ with
\[
\Delta_1(x_2)<-\frac{x_2^{1-K\eta(\gamma_0)}}{\gamma_0^{2+\epsilon}\log x_2}.
\]
Taking logarithms we find that
\begin{align*}
    \log (-\Delta_1(x_2))>&\log x_2-K\eta(\gamma_0)(1+\epsilon)\log X_0 -(2+\epsilon)\log \gamma_0-\log\log x_2\\
    =&\log x_2-K(1+\epsilon)\frac{\eta(\gamma_0)}{|\eta'(\gamma_0)|\gamma_0}-(2+\epsilon)\log \gamma_0-\log\log x_2.
\end{align*}
On the other hand, as a consequence of Lemma~\ref{lemma int E_2^2 by zero-free region} and assumption (i), 
\begin{align*}
    \log \Delta_2(x_2)\leq &\log x_2-2(1-\epsilon)\omega(x_2)+O_\epsilon(1)-2\log\log x_2\\
    \leq& \log x_2-2(1-\epsilon)\omega(X_0)-\log\log x_2\\
    =& \log x_2-2(1-\epsilon)\left(\frac{\eta(\gamma_0)}{|\eta'(\gamma_0)|\gamma_0}+\log \gamma_0\right)-\log\log x_2,
\end{align*}
where we used the fact that $\omega(x)$ is an increasing function. Therefore, we would have $-\Delta_1(x_2)>\Delta_2(x_2)$ if
\[
(2(1-\epsilon)-K(1+\epsilon))\frac{\eta(\gamma_0)}{|\eta'(\gamma_0)|\gamma_0}>3\epsilon \log\gamma_0.
\]
Since $K<2$, assumption (iv) guarantees that we can choose $\epsilon$ (depending only on $K$, not $\gamma_0$) small enough such that this inequality is verified.

Finally, the infinitude of such zeros $\rho_0$ yields arbitrarily large $x_2$ with $-\Delta_1(x_2)>\Delta_2(x_2)$, thereby concluding the proof.

\section{Real primitive characters: Proof of Theorem~\ref{thm real prim char}}\label{section: real prim char}

We deduce by partial summation that 
\begin{align*}
    E_1(x;d)=&-\sum_{p>x}\frac{\chi_d(p)\log p}{p}=\frac{\theta(x,\chi_d)}{x}-\int_x^\infty \frac{\theta(y,\chi_d)}{y^2}\md y,\\
    E_2(x;d)=&-\sum_{p>x}\frac{\chi_d(p)}{p}=\frac{\pi(x,\chi_d)}{x}-\int_x^\infty \frac{\pi(y,\chi_d)}{y^2}\md y, 
\end{align*}
from which it follows that
\begin{align*}
    \int_2^X E_1(x;d)\md x=&XE_1(X;d)-\theta(X,\chi_d)+2\mc{E}_{1}(d)=-X\int_X^\infty \frac{\theta(x,\chi_d)}{x^2}\md x+2\mc{E}_{1}(d),\\
    \int_2^X E_2(x;d)\md x=&XE_2(X;d)-\pi(X,\chi_d)+2\mc{E}_{2}(d)=-X\int_X^\infty \frac{\pi(x,\chi_d)}{x^2}\md x+2\mc{E}_{2}(d).
\end{align*}

We only focus on $E_1$ since the case for $E_2$ can be treated similarly. From \eqref{explicit formula for psi(x,chi)} and \eqref{error psi(x,chi)-theta(x,chi)} we have
\begin{equation}\label{explicit formula of int E_1(x;d)}
    \int_2^X E_1(x;d)\md x=-\sum_{\rho} \frac{X^{\rho}}{\rho(\rho-1)}+2\sqrt{X}+o(\sqrt{X}).
\end{equation}
Assuming GRH for $L(s,\chi_d)$, $\int_2^X E_1(x;d)\md x>0$ for all sufficiently large $X$ if 
\begin{equation}\label{sum over gamma_d <2}
    B_{\chi_d}:=\sum_{\gamma}\frac{1}{1/4+\gamma^2}<2. 
\end{equation}
In general $\Re \frac{\xi'}{\xi}(0,\chi)=-\sum_{\rho} \Re(1/\rho)=-B_\chi/2$ (see \cite[p. 81]{Dav}) for an arbitrary primitive character $\chi$, so here we have $B_{\chi_d}=-2\frac{\xi'}{\xi}(0,\chi_d)$ because $\chi_d$ is real. The relation between $B_{\chi_d}$ and the logarithmic derivative of $L(s,\chi_d)$ is given by (see \cite[(10.39)]{MV})
\begin{equation}\label{B in terms of L'/L}
    B_{\chi_d}:=\log \frac{|d|}{\pi}+2\frac{L'}{L}(1,\chi_d)-\gamma_E-2(1-\kappa)\log 2,
\end{equation}
where $\kappa=0$ or 1 according as $\chi_d(-1)=1$ (i.e., $d>0$) or $\chi_d(-1)=-1$ (i.e., $d<0$).

Observe that if $B_{\chi_d}>2$ when we assume both GRH and LI for $L(s,\chi_d)$, then the left-hand side of \eqref{explicit formula of int E_1(x;d)} is $\Omega_{\pm}(\sqrt{X})$ in view of Kronecker's theorem. Hence, conditional on LI, a fundamental discriminant $d$ satisfies the equivalence in Theorem~\ref{thm E_3} if and only if it satisfies \eqref{sum over gamma_d <2} (barring the extremely unlikely case where $B_{\chi_d}$ equals exactly 2, but we shall rule out this possibility by computation). Denote by $\mc{D}$ the set of all such $d$'s, which is finite since the number of non-trivial zeros of $L(s,\chi_d)$ up to a fixed height $T$ becomes arbitrarily large as $d\to \infty$ (at least for $T\geq 5/7$; see Lemma~\ref{N(T,chi)} below). To determine this set, we use the following explicit bound on the zero-counting function due to Bennett \textit{et al.}:

\begin{lemma}[{\cite[Theorem 1.1]{BMOR}}]\label{N(T,chi)}
    Let $\chi$ be a character with conductor $q>1$. If $T\geq 5/7$ and $\ell:=\log \frac{q(T+2)}{2\pi}>1.567$, then
    \[
    \bigg|N(T,\chi)-\left(\frac{T}{\pi}\log \frac{qT}{2\pi e}-\frac{\chi(-1)}{4}\right)\bigg|\leq 0.22737\ell+2\log(1+\ell)-0.5,
    \]
where $N(T,\chi)=\#\{\rho: L(\rho,\chi)=0, \: 0<\beta<1, \:|\gamma|\leq T \}$.
\end{lemma}

Considering the symmetry of zeros of $L(s,\chi_d)$, one can calculate the partial sum of the defining series \eqref{sum over gamma_d <2} for $B_{\chi_d}$ with the first $N$ positive ordinates $\gamma_1,\ldots,\gamma_N$ (we may assume that $L(1/2,\chi_d)\neq 0$, otherwise \eqref{sum over gamma_d <2} is plainly false):
\begin{align*}
    B_{\chi_d}=&\sum_{i=1}^N \frac{2}{1/4+\gamma_i^2}+\int_{{\gamma_{N+1}}^{-}}^\infty \frac{\mr{d}N(t,\chi_d)}{1/4+t^2}\\
    =&\sum_{i=1}^N \frac{2}{1/4+\gamma_i^2}-\frac{2N}{1/4+\gamma_{N+1}^2}+\int_{\gamma_{N+1}}^\infty \frac{2tN(t,\chi_d)}{(1/4+t^2)^2}\md t.
\end{align*}
This works if $\gamma_N\neq \gamma_{N+1}$, otherwise a slight modification suffices. The sum can be evaluated as long as the first $N$ zeros have been computed to sufficient accuracy, and the tail integral can be bounded by Lemma~\ref{N(T,chi)}. Hence, if $N$ is sufficiently large we would have enough precision to determine whether $d\in \mc{D}$ or not (again barring the coincidence $B_{\chi_d}=2$).

We now narrow down the range of $d$ to be checked. With the same notations and conditions as in Lemma~\ref{N(T,chi)}, take $\widetilde{N}(T,\chi_d)$ to be the least non-negative even integer greater than or equal to
\[
\frac{T}{\pi}\log \frac{|d|T}{2\pi e}-0.25-0.22737\ell-2\log(1+\ell)+0.5.
\]
Then $N(T,\chi_d)\geq \widetilde{N}(T,\chi_d)$ for $T\geq 5/7$, and hence
\[
B_{\chi_d}\geq \int_{T}^\infty \frac{2t\widetilde{N}(t,\chi_d)}{(1/4+t^2)^2}\md t.
\]
With the choice $T=0.928$, a numerical calculation shows that $d\not \in \mc{D}$ if $|d|\geq 460,000$. We used LCALC \cite{Rub} on SageMath \cite{Sage} to check all $d$ up to this magnitude, which took under three days on a 3.5GHz PC. A total of 178 fundamental discriminants were found to belong to $\mc{D}$, of which 125 are positive (with the largest being $1201$) and 53 are negative (the smallest being $-551$) \footnote{Note the significant disparity between the number of $d$'s we have to check and the actual cardinality of $\mc{D}$. A more efficient explicit estimate on $N(T,\chi)$ for small $T$ would probably have spared us from checking such a long range.}. They are given in Tables~\ref{table d>0} and ~\ref{table d<0} along with the approximate values of their corresponding $B_{\chi_d}$. There are quite fewer negative ones as it can be seen from \eqref{B in terms of L'/L} that $B_{\chi_d}$ tends to be larger when $\chi_d$ is odd. Every $d$ from $-40$ to $149$ is in $\mc{D}$, and every $d$ with $|d|\leq 100$ is in $\mc{D}$ except for $d\in\{-91, -88, -67, -43\}$. 

The direction $(b)\implies (a)$ in Theorem~\ref{thm real prim char} is obtained by modifying the argument in \S\ref{section: proof of thm E_1 E_2 necessity}. Supposing that GRH fails for some real $\chi$, we want to show that
\[
\int_X^\infty \frac{\psi(x,\chi)}{x^2}\md x=\Omega_{\pm}(X^{\Theta_\chi-1-\epsilon}),
\]
so that $\int_2^X E_1(x;d)\md x=\Omega_{\pm}(X^{\Theta_\chi-\epsilon})$. It is worth pointing out that for the proof to work as before, we need the right-hand side of
\[
\int_1^\infty \frac{\psi(x,\chi)}{x^{s+1}}\md x=-\frac{L'(s,\chi)}{sL(s,\chi)}
\]
to be analytic at all real $s>\Theta_{\chi}-\epsilon$. This is true as long as $L(\Theta_{\chi},\chi)\neq 0$, since then there must exist some $\epsilon>0$ such that $L(s,\chi)$ does not vanish on $(\Theta_\chi-\epsilon,\Theta_\chi)$. In particular it is true \footnote{In fact, Platt \cite[Theorems 7.1 and 7.2]{Platt} has verified that $L(s,\chi)$ has no real zero in $(0,1)$ for every primitive character $\chi$ of modulus $q\leq 400,000$.} for all $d\in \mc{D}$.

\section{Arithmetic progressions: Proof of Theorem~\ref{thm arith prog}} \label{section: arith prog}

We have
\begin{align*}
    E_1(x;q,a)=&\frac{\theta(x;q,a)-x/\phi(q)}{x}-\int_x^\infty \frac{\theta(y;q,a)-y/\phi(q)}{y^2}\md y,\\
    E_2(x;q,a)=&\frac{\pi(x;q,a)-\li(x)/\phi(q)}{x}-\int_x^\infty \frac{\pi(y;q,a)-\li(y)/\phi(q)}{y^2}\md y,
\end{align*}
and
\begin{align*}
    \int_2^X E_1(x;q,a)\md x=&-X\int_X^\infty \frac{\theta(x;q,a)-x/\phi(q)}{x^2}\md x+\widetilde{\mc{E}}_{1}(q,a),\\
    \int_2^X E_2(x;q,a)\md x=&-X\int_X^\infty \frac{\pi(x;q,a)-\li(x)/\phi(q)}{x^2}\md x+\widetilde{\mc{E}}_{2}(q,a)
\end{align*}
where $\widetilde{\mc{E}}_{i}(q,a)$ are appropriate constants that do not affect our discussion. Again we only address $E_1$ here. On combining \eqref{explicit formula for psi(x,chi)}, \eqref{orthogonality} and \eqref{error psi-theta arith pro} we see that
\[
    \int_2^X E_1(x;q,a)\md x=-\frac{1}{\phi(q)}\sum_{\chi \ (\mr{mod}\ q)}\ov{\chi}(a)\sum_{\rho_\chi}\frac{X^{\rho_\chi}}{\rho_\chi(\rho_\chi-1)}+\frac{2C(q,a)\sqrt{X}}{\phi(q)}+o(\sqrt{X}).
\]
Suppose that GRH holds for all $L(s,\chi)$ modulo $q$. If
\begin{equation}\label{def B_q}
    B_q:=\sum_{\chi \ (\mr{mod}\ q)}\sum_{\rho_\chi}\frac{1}{1/4+\gamma_\chi^2}<2C(q,a),
\end{equation}
then $\int_2^X E_1(x;q,a)>0$ for all sufficiently large $X$, and the set of moduli $q$ satisfying this bound is again finite. 

For $a=1$ we verify that this set is exactly $\mc{Q}$ as defined in \eqref{set Q}. Let $\omega(q)$ denote the number of distinct prime factors of $q$ and $\nu_2(q)$ the exponent of the highest power of $2$ dividing $q$. Then
\[
C(q,1)=
\begin{cases}
    2^{\omega(q)} & \nu_2(q) = 0 \:\:\text{or} \:\:2,\\
    2^{\omega(q)-1} & \nu_2(q) = 1,\\
    2^{\omega(q)+1} & \nu_2(q) \geq 3.
\end{cases}
\]
Moreover, the number of primitive characters modulo $q$ is given by 
\[
n(q)=q\prod_{p\mathrel\Vert q}(1-2/p)\prod_{p^2\mid q}(1-1/p)^2.
\]
Take $T=4$, so that $\ell>1.567$ for any $q\geq 6$ where $\ell$ was defined in Lemma~\ref{N(T,chi)}. As a result, if $q\not \equiv 2\ (\mr{mod} \ 4)$ (otherwise $n(q)=0$ and $B_q=B_{q/2}$), then for some primitive character $\chi \ (\mr{mod}\ q)$,
\begin{align}
    B_q> &\:
    n(q)\sum_{\gamma_\chi} \frac{1}{1/4+\gamma_\chi^2}\notag\\
    =&\: n(q)\int_0^\infty \frac{2tN(t,\chi)}{(1/4+t^2)^2}\md t \notag\\
    \geq &\: n(q)\int_4^\infty \frac{2t(\frac{t}{\pi}\log \frac{qt}{2\pi e}-0.25-0.22737\ell-2\log(1+\ell)+0.5)}{(1/4+t^2)^2}\md t \label{lower bound B_q}.
\end{align}
A short calculation shows that for all non-trivial $q>2040$, $n(q)/C(q,1)>3$ and the integral on the right-hand side of \eqref{lower bound B_q} is greater than $2/3$, and thus $B_q>2C(q,1)$. We checked all $q$ up to 2040 to obtain the candidate set $\mc{Q}'$ consisting of all elements $q'$ for which the right-hand side of \eqref{lower bound B_q} is bounded above by $2C(q',1)$. Then, for each $q\in \mc{Q}'$ we used LCALC to compute a partial sum of $B_q$ with the first $700$ zeros of each $\chi \ (\mr{mod}\ q)$ and bounded the tail using Lemma~\ref{N(T,chi)}, which ruled out all $q\in \mc{Q}'\setminus \mc{Q}$. 

Conversely, suppose that GRH fails for some character modulo $q$ so that $\Theta_q>1/2$. We need
\[
\int_1^\infty \frac{\psi(x;q,a)-x/\phi(q)}{x^{s+1}}\md x=-\frac{1}{\phi(q)(s-1)}-\frac{1}{\phi(q)}\sum_{\chi \ (\mr{mod}\ q)}\ov{\chi}(a)\frac{L'(s,\chi)}{sL(s,\chi)}
\]
to be analytic at all real $s>\Theta_q-\epsilon$ but not on the half plane $\sigma> \Theta_q-\epsilon$. This is true if $\prod_{\chi \ (\mr{mod}\ q)} L(\Theta_q,\chi)\neq 0$ (which holds for all $q\in \mc{Q})$ and $a=1$. Indeed, in this case even if multiple characters modulo $q$ share the same zero $\rho$, their contributions to the pole at $s=\rho$ do not cancel out since $\Re(1/\rho)>0$, while if $a\neq 1$ the rotations caused by $\ov{\chi}(a)$ could possibly remove the pole.

\section{Concluding remarks and further questions}
\begin{enumerate}
    \item 
    Unsatisfactorily, the proof of Theorem~\ref{thm E_3} relies on Assumption~\ref{assumption} in the case where $\Theta=1$. Is it possible to further simplify, or even completely remove this assumption?
    
    \item 
    Theorems~\ref{thm real prim char} and \ref{thm arith prog} are closely related to the following (perhaps more natural) questions. Assuming GRH, for which $d$ is
    \[
    \int_2^X \pi(x,\chi_d)\md x<0, \hspace{0.5cm} X>X_0
    \]
    true? For which pairs of $(q,a)$ is
    \begin{equation}\label{(q,a)}
        \int_2^X \left(\pi(x;q,a)-\frac{\li(x)}{\phi(q)}\right)\md x<0, \hspace{0.5cm} X>X_0
    \end{equation}
    true\footnote{This is Problem \#23 in \cite{HKMB}, posed by Johnston.}? Considering the integral of the explicit formula, we see that the answer to the first question is some set $\mc{D}'$ containing $\mc{D}$, essentially because $|\rho(\rho+1)|>|\rho(\rho-1)|$. We slightly modified the code and found 381 elements in $\mc{D}'$, the smallest being $-1151$ and largest 2161. We also addressed the second question as we did in \S\ref{section: arith prog} and listed the results in Table \ref{table (q,a)}. 

    \begin{table} 
    \centering 
    \caption{Pairs $(q,a)$ for which \eqref{(q,a)} holds.}
    \label{table (q,a)}
    \begin{tabular}{ c | c c c c c}
         $q$ & & & $a$\\
         \hline
         2 & 1\\
         \hline
         3 & 1 \\
         \hline
         4 & 1 \\
         \hline
         5 & 1 & 4 \\ 
         \hline
         7 & 1 & 2 & 4\\
         \hline
         8 & 1\\
         \hline 
         9 & 1 & 4 & 7 \\  
         \hline 
         10 & 1 & 9\\
         \hline
         11 & 1 & 3 & 4 & 5 & 9 \\
         \hline 
         12 & 1\\
         \hline
         14 & 1 & 9 & 11\\
         \hline 
         15 & 1 & 4\\
         \hline 
         16 & 1 & 9 \\
         \hline
         18 & 1 & 7 & 13 \\
         \hline
         20 & 1 & 9\\
         \hline
         21 & 1 & 4 & 16\\
         \hline 
         22 & 1 & 3 & 5 & 9 & 15\\
         \hline 
         24 & 1\\
         \hline
         28 & 1 & 9 & 25\\
         \hline 
         30 & 1 & 19\\
         \hline
         36 & 1 & 13 & 25\\
         \hline 
         40 & 1 & 9\\
         \hline 
         42 & 1 & 25 & 37\\
         \hline 
         48 & 1 & 25\\
         \hline 
         60 & 1 & 49\\
         \hline 
         84 & 1 & 25& 37\\
         \hline 
         120 & 1 & 49 \\
         \hline
    \end{tabular}
    \end{table}

    \item One may also look at the analogue of Mertens' theorems for number fields $\bb{K}/\bb{Q}$ (see, e.g., \cite{Ros}). 
    For example, for a quadratic field $\bb{K}=\bb{Q}(\sqrt{d})$ we know that $\zeta_\bb{K}(s)=\zeta(s)L(s,\chi_{\Delta_\bb{K}})$ where $\Delta_{\bb{K}}$ is the discriminant of $\bb{K}$. Thus, if $B_1+B_{\chi_{\Delta_\bb{K}}}<2$, then RH for $\zeta_\bb{K}(s)$ is equivalent to the condition
    \[
    \int_2^X \bigg(\sum_{N(\mathfrak{p})\leq x}\frac{1}{N(\mathfrak{p})}-\log\log x-\mc{E}_{2,\bb{K}}\bigg)\md x >0, \hspace{0.5cm} X> X_0.
    \]
    As the degree or the discriminant of $\bb{K}$ go up, so does the number of non-trivial zeros of $\zeta_{\bb{K}}(s)$ up to height $T$. According to the Hermite\textendash Minkowski theorem we can only assert this equivalence for finitely many $\bb{K}$ (up to isomorphism). Can one classify all such $\bb{K}$?
\end{enumerate}

\begin{table}[ht!]
\centering
  \caption{Positive fundamental discriminants in $\mc{D}$ and their associated $B_{\chi_d}$.}
  \label{table d>0}
 \begin{tabular}{|c c | c c| c c | c c |} 
 \hline 
 $d$ & $B_{\chi_d}$ & $d$ & $B_{\chi_d}$  & $d$ & $B_{\chi_d}$ & $d$ & $B_{\chi_d}$ \\
 \hline
 5 & $0.156\ldots$ & 8 & $0.235\ldots$ & 12 & $0.330\ldots$ & 13 & $0.396\ldots$  \\
 \hline 
 17 & $0.387\ldots$ & 21 & $0.614\ldots$ & 24 & $0.552\ldots$ & 28 & $0.572\ldots$  \\
 \hline 
 29 & $0.879\ldots$ & 33 & $0.566\ldots$ & 37& $0.814\ldots$ & 40 & $0.697\ldots$  \\
 \hline 
 41 & $0.682\ldots$ & 44 & $0.908\ldots$ & 53 & $1.463\ldots$ & 56 & $1.056\ldots$ \\
 \hline
 57 & $0.768\ldots$ & 60 & $0.973\ldots$ & 61 & $0.993\ldots$ & 65 & $0.897\ldots$ \\
 \hline
 69 & $1.238\ldots$ & 73 & $0.797\ldots$ & 76 & $0.939\ldots$ & 77 & $1.922\ldots$\\
 \hline
 85 & $1.176\ldots$ & 88 & $1.070\ldots$ & 89 & $1.033\ldots$ & 92 & $1.565\ldots$ \\
 \hline
 93 & $1.590\ldots$ & 97 & $0.913\ldots$ & 101 & $1.838\ldots$ & 104 & $1.498\ldots$ \\
 \hline 
 105 & $1.014\ldots$ & 109 & $1.218\ldots$ & 113 & $1.245\ldots$ & 120 & $1.368\ldots$\\
 \hline
 124 & $1.164\ldots$ & 129 & $1.079\ldots$ & 133 & $1.573\ldots$ & 136 & $1.210\ldots$ \\
 \hline
 137 & $1.368\ldots$ & 140 & $1.853\ldots$ & 141 & $1.555\ldots$ & 145 & $1.052\ldots$ \\
 \hline
 149 & $1.926\ldots$ & 156 & $1.418\ldots$ & 157 & $1.769\ldots$ & 161 & $1.328\ldots$ \\
 \hline
 165 & $1.824\ldots$ & 168 & $1.783\ldots$ & 172 & $1.449\ldots$ & 177 & $1.315\ldots$\\
 \hline
 181 & $1.557\ldots$ & 184 & $1.302\ldots$ & 185 & $1.543\ldots$ & 193 & $1.195\ldots$ \\
 \hline
 201 & $1.291\ldots$ & 204 & $1.563\ldots$ & 205 & $1.656\ldots$ & 209 & $1.511\ldots$ \\
 \hline
 217 & $1.278\ldots$ & 220 & $1.585\ldots$ & 229 & $1.714\ldots$ & 232 & $1.620\ldots$ \\
 \hline
 233 & $1.746\ldots$ & 241 & $1.228\ldots$ & 249 & $1.338\ldots$ & 253 & $1.969\ldots$ \\
 \hline
 257 & $1.962\ldots$ & 264 & $1.868\ldots$ & 265 & $1.331\ldots$ & 268 & $1.709\ldots$ \\
 \hline
 273 & $1.617\ldots$ & 280 & $1.667\ldots$ & 281 & $1.587\ldots$ & 301 & $1.798\ldots$ \\
 \hline
 305 & $1.781\ldots$ & 309 & $1.964\ldots$ & 313 & $1.501\ldots$ & 316 & $1.563\ldots$ \\
 \hline
 321 & $1.556\ldots$ & 329 & $1.735\ldots$ &
 337 & $1.468\ldots$ & 345 & $1.581\ldots$\\
 \hline
 364 & $1.677\ldots$ & 376 & $1.809\ldots$ & 385 & $1.490\ldots$ & 393 & $1.812\ldots$ \\
 \hline
 401 & $1.747\ldots$ & 409 & $1.489\ldots$ & 417 & $1.852\ldots$ & 421 & $1.890\ldots$ \\
 \hline
 424 & $1.739\ldots$ & 433 & $1.689\ldots$ &
 449 & $1.808\ldots$ & 456 & $1.974\ldots$ \\
 \hline
 457 & $1.648\ldots$ & 465 & $1.889\ldots$ & 481 & $1.570\ldots$ & 489 & $1.790\ldots$ \\
 \hline
 505 & $1.596\ldots$ & 520 & $1.937\ldots$ &
 721 & $1.747\ldots$ & 745 & $1.940\ldots$ \\
 \hline
 769 & $1.840\ldots$ & 793 & $1.916\ldots$ &
 849 & $1.963\ldots$ & 865 & $1.894\ldots$ \\
 \hline
 889 & $1.808\ldots$ & 1009 & $1.865\ldots$ & 1081 & $1.988\ldots$ & 1129 & $1.924\ldots$ \\ 
 \hline
 1201 & $1.950\ldots$ & & & & & &\\
 \hline
 \end{tabular}
\end{table}

\clearpage

\begin{table}[ht!]
\centering
\caption{Negative fundamental discriminants in $\mc{D}$ and their associated $B_{\chi_d}$.}
\label{table d<0}
 \begin{tabular}{|c c | c c| c c | c c |} 
 \hline 
 $d$ & $B_{\chi_d}$ & $d$ & $B_{\chi_d}$  & $d$ & $B_{\chi_d}$ & $d$ & $B_{\chi_d}$ \\
 \hline
 $-3$ & $0.113\ldots$ & $-4$ & $0.155\ldots$ & $-7$ & $0.255\ldots$ & $-8$ & $0.316\ldots$ \\
 \hline
 $-11$ & $0.507\ldots$ & $-15$ & $0.459\ldots$ &
 $-19$ & $1.052\ldots$ & $-20$ & $0.638\ldots$ \\
 \hline
 $-23$ & $0.569\ldots$ & $-24$ & $0.797\ldots$ & $-31$ & $0.809\ldots$ & $-35$ & $1.110\ldots$ \\
 \hline
 $-39$ & $0.823\ldots$ & $-40$ & 
 $1.404\ldots$ & $-47$ & $0.845\ldots$ &
 $-51$ & $1.582\ldots$ \\
 \hline
 $-52$ & $1.837\ldots$ & $-55$ & $1.175\ldots$ & $-56$ & $1.096\ldots$ & $-59$ & $1.289\ldots$ \\
 \hline
 $-68$ & $1.335\ldots$ & $-71$ & $0.975\ldots$ & $-79$ & $1.382\ldots$ & $-83$ & $1.778\ldots$ \\
 \hline
 $-84$ & $1.576\ldots$ & $-87$ & $1.330\ldots$ & $-95$ & $1.163\ldots$ & $-103$ & $1.774\ldots$ \\
 \hline
 $-104$ & $1.389\ldots$ & $-111$ & $1.300\ldots$ & $-116$ & $1.540\ldots$ &
 $-119$ & $1.190\ldots$ \\
 \hline
 $-131$ & $1.677\ldots$ & $-143$ & $1.406\ldots$ & $-151$ & $1.853\ldots$ & $-152$ & $1.943\ldots$ \\
 \hline
 $-159$ & $1.485\ldots$ & $-164$ & $1.617\ldots$ & $-167$ & $1.489\ldots$  &
 $-191$ & $1.468\ldots$ \\
 \hline
 $-199$ & $1.870\ldots$ & $-215$ & $1.516\ldots$ & $-231$ & $1.755\ldots$ &
 $-239$ & $1.579\ldots$ \\
 \hline
 $-255$ & $1.892\ldots$ & $-263$ & $1.914\ldots$ & $-287$ & $1.933\ldots$ & $-311$ & $1.615\ldots$ \\
 \hline
 $-335$ & $1.783\ldots$ & $-359$ & $1.818\ldots$ & $-431$ & $1.938\ldots$ & $-479$ & $1.824\ldots$ \\
 \hline
 $-551$ & $1.980\ldots$ & & & & & &\\
 \hline
 \end{tabular}
\end{table}

\section{Data Availability Statement}
The SageMath code used to generate the relevant results is available at \cite{Git}.

\section{Acknowledgments}
I am grateful to Professor Ghaith Hiary for many fruitful discussions on this topic, and also to the anonymous referee for their invaluable comments and suggestions.

\sloppy
\printbibliography

\end{document}